\documentclass[10pt,a4paper]{amsart}%
\usepackage{amsmath, amsfonts, amsthm,color, latexsym, amssymb}
\usepackage{amscd}
\usepackage{amsmath}
\usepackage{amsfonts}
\usepackage{amssymb}
\usepackage[all]{xy}
\usepackage{graphicx}%
\providecommand{\U}[1]{\protect\rule{.1in}{.1in}}
\providecommand{\U}[1]{\protect\rule{.1in}{.1in}}
\newtheorem{theorem}{Theorem}[section]
\newtheorem{proposition}[theorem]{Proposition}
\newtheorem{corollary}[theorem]{Corollary}

\newtheorem{remark}[theorem]{Remark}

\newtheorem{lemma}[theorem]{Lemma}
\newtheorem{definition}[theorem]{Definition}
\begin{document}
\title{Weak compactness and strongly summing multilinear operators}
\author{Daniel Pellegrino, Pilar Rueda, Enrique A. S\'anchez-P\'erez}
\thanks{D. Pellegrino acknowledges with thanks the support of CNPq Grant 313797/2013-7 (Brazil). P. Rueda acknowledges with thanks the support of the Ministerio de
Econom\'{\i}a y Competitividad (Spain) MTM2011-22417. E.A. S\'anchez P\'erez
acknowledges with thanks the support of the Ministerio de Econom\'{\i}a y
Competitividad (Spain) MTM2012-36740-C02-02.}
\date{}
\maketitle

\begin{abstract}
Every absolutely summing linear operator is weakly compact. However, for
strongly summing multilinear operators and polynomials -- one of the most
natural extensions of the linear case to the non linear framework -- weak
compactness does not hold in general. We show that a subclass of the class of
strongly summing multilinear operators/polynomials, sharing its main
properties such as Grothendieck's Theorem, Pietsch Domination Theorem and
Dvoretzky--Rogers Theorem, has even better properties like weak compactness
and a natural factorization theorem.

\end{abstract}

\section{Introduction}

The theory of absolutely summing linear operators has its roots in the 1950s
with A. Grothendieck's pioneer ideas; in its modern presentation, it appeared
in 1966-67 in the works of A. Pietsch \cite{pi2} and B. Mitiagin and A.
Pe\l czy\'{n}ski \cite{miti}. A cornerstone in the theory is the remarkable
paper of J. Lindenstrauss and A. Pe\l czy\'{n}ski \cite{lp}, which clarified
Grothendieck's ideas, without the use of tensor products. Lindenstrauss and
Pe\l czy\'{n}ski were also responsible for the reformulation of Grothendieck's
inequality, which is still a fundamental result of Banach Space Theory and
Mathematical Analysis in general (see \cite{pisier}). Nowadays absolutely
summing operators is a current subject in books of Banach Space Theory (see,
for instance, \cite{albiac, varios}). For a detailed approach to the linear
theory of absolutely summing operators we refer to the excellent book of J.
Diestel, H. Jarchow and A. Tonge \cite{DJT}.

It is then comprehensive that a big effort has been made, since Pietsch's
proposal \cite{pi}, to try and generalize the linear theory to non linear
operators. Many families of non linear operators have been considered such as
multilinear operators, homogeneous polynomials, holomorphic mappings, $\alpha
$-homogeneous mappings, Lipschitz mappings among others. However, extending
summability properties to non linear operators has been proved difficult and
intriguing. For instance, there are several extensions of absolutely
$p$-summing linear operators to the multilinear setting that have been
considered in the literature. Besides its intrinsic interest, the multilinear
theory of absolutely summing operators has shown important connections,
including applications to Quantum Information Theory (see \cite{monta}). This
proliferation of classes of summing multilinear maps has lead to the
appearance of works that compare different approaches (see \cite{CD, Pe}). The
first challenging task when dealing with multilinear operators is probably to
identify the class of multilinear operators that best inherits the spirit of
the absolutely summing linear operators. According to \cite{pp0, pp3} one of
the most natural extensions of the notion of absolutely $p$-summing linear
operators to the multilinear setting is the notion of strongly $p$-summing
multilinear operators, due to V. Dimant (\cite{dimant}). This class lifts to
the multilinear framework most of the main properties of absolutely
$p$-summing linear operators: Grothendieck's Theorem, Pietsch Domination
Theorem, Inclusion Theorem. However, as we will see, a natural version of the
Pietsch Factorization Theorem does not hold for this class.

The good behavior of multilinear extensions has found no echo when considering
extensions of absolutely summing operators to polynomials. In this non linear
setting, several attempts have been made but all of them have found rough
edges to succeed in. This is the case of $p$-dominated homogeneous
polynomials, for which a Pietsch type factorization theorem has been pursuit
(see \cite{anais,MeTo, BoPeRu2007,correc}) and succeeded just when the domain
is separable. Recently, the second and third authors \cite{RuSaFact} have
isolate the class of $p$-dominated polynomials that satisfy a Pietsch type
factorization theorem: the factorable $p$-dominated polynomials. However, even
if this makes a big difference with $p$-dominated polynomials, they still lack
good properties as evidenced by the fact that factorable $p$-dominated
polynomial do not define a composition ideal or,
equivalently, the linearization of a factorable $p$-dominated polynomial may
not be absolutely $p$-summing.

Our aim in this paper is to introduce factorable strongly $p$-summing
multilinear operators and homogeneous polynomials to the full extent of
absolutely $p$-summing linear operators. These new classes of summing
polynomials/multilinear operators stand apart from previous generalizations as
they keep a big amount of the fundamental properties  that are satisfied in the linear theory and are not satisfied by former non linear classes.
Factorable strongly $p$-summing multilinear operators is a subclass of
strongly $p$-summing multilinear operators that has in addition a quite
natural Pietsch Factorization type theorem and weak compactness. Factorable
strongly $p$-summing homogeneous polynomials also fulfills a factorization
theorem in the spirit of Pietsch, are weakly compact and a polynomial belongs
to the class if and only if its second adjoint (in the sense of Aron and
Schottenloher) is in the class. Actually, an homogeneous polynomial is
factorable strongly $p$-summing if and only if its associated multilinear map
is factorable $p$-summing or, equivalently, its linearization is absolutely
$p$-summing. This brings deep strengths that are not shared by former classes
of summing polynomials as dominated or strongly summing polynomials.

This paper is organized as follows. The next section contains the basics
(definitions and main results) on linear and non linear summability that are
in order for our purposes. In Section \ref{multilinear} we show that a slight
modification of the notion of strongly $p$-summing operator (inspired in a
recent paper of the second and third author) generates a subclass that keeps
its main properties and also has a factorization theorem in the lines of the
approach above. These are the factorable strongly $p$-summing multilinear
operators. As a consequence we have weak compactness, as in the linear case.
In Section \ref{polynomials} we deal with homogeneous polynomials, proving
that a polynomial is factorable strongly $p$-summing if and only if its
linearization is absolutely $p$-summing. The connection between $m$%
-homogeneous polynomials and $m$-linear operators is established: an
$m$-homogeneous polynomial is factorable strongly $p$-summing if and only if
its associated symmetric $m$-linear map is factorable strongly $p$-summing.
These results yield to obtain in Section \ref{wealth} proper generalizations
of fundamental properties related to summability for linear operators to
multilinear maps and homogeneous polynomials. Among other results, we show
that a Dvoretzky-Rogers type theorem, a Lindenstrauss--Pe\l czy\'nski type
theorem or a Grothendieck type theorem work for factorable strongly
summability. Finally, in Section \ref{coherence} we show that the sequence
formed by the ideals of factorable strongly summing homogeneous polynomials
and factorable strongly summing multilinear operators is coherent and
compatible with the ideal of absolutely summing linear operators.

\section{Background: linear and multilinear summability}

If $1\leq p<\infty$ and $X,Y$ are Banach spaces, a continuous linear operator
$u:X\rightarrow Y$ is absolutely $p$-summing $\left(  u\in\Pi_{p}\left(
X;Y\right)  \right)  $ if there is a constant $C\geq0$ such that
\[
\left(
{\displaystyle\sum\limits_{j=1}^{m}}
\left\Vert u(x_{j})\right\Vert ^{p}\right)  ^{1/p}\leq C\left(  \sup
_{\varphi\in B_{X^{\ast}}}\sum\limits_{j=1}^{m}\left\vert \varphi
(x_{j})\right\vert ^{p}\right)  ^{1/p}%
\]
for all $x_{1},...,x_{m}\in X$ and all positive integers $m$. The infimum of
all $C$ that satisfy the above inequality defines a norm, denoted by $\pi
_{p}(u),$ and $\left(  \Pi_{p}\left(  X,Y\right)  ,\pi_{p}\right)  $ is a
Banach space. The cornerstones of the theory of absolutely summing linear
operators are the following theorems:

\begin{itemize}
\item (Dvoretzky-Rogers theorem) If$\ p\geq1,$ then $\Pi_{p}(X;X)=\mathcal{L}%
(X;X)$ if and only if $\dim X<\infty.$

\item (Grothendieck's theorem ) Every continuous linear operator from
$\ell_{1}$ to $\ell_{2}$ is absolutely $1$-summing.

\item (Lindenstrauss--Pe\l czy\'{n}ski theorem) If $X$ and $Y$ are
infinite-dimensional Banach spaces, $X$ has an unconditional Schauder basis
and $\Pi_{1}(X;Y)=\mathcal{L}(X;Y)$ then $X=\ell_{1}$ and $Y$ is a Hilbert space.

\item (Pietsch Domination theorem) If $X$ and $Y$ are Banach spaces, a
continuous linear operator $u:X\rightarrow Y$ is absolutely $p$-summing if and
only if there exist a constant $C\geq0$ and a Borel probability measure $\mu$
on the closed unit ball of the dual of $X,$ $\left(  B_{X^{\ast}}%
,\sigma(X^{\ast},X)\right)  ,$ such that%
\begin{equation}
\left\Vert u(x)\right\Vert \leq C\left(  \int_{B_{X^{\ast}}}\left\vert
\varphi(x)\right\vert ^{p}d\mu\right)  ^{\frac{1}{p}} \label{gupdt}%
\end{equation}
for all $x\in X.$

\item (Inclusion theorem) If $1\leq p\leq q<\infty,$ then every absolutely
$p$-summing operator is absolutely $q$-summing.

\item (Pietsch Factorization theorem) A continuous linear operator
$u:X\rightarrow Y$ is absolutely $p$-summing if, and only if, there exist a
regular Borel probability measure $\mu$ on $B_{X^{\ast}}$, a closed subspace
$X_{p}$ of $L_{p}\left(  \mu\right)  $ and a continuous linear operator
$\widehat{u}:X_{p}\rightarrow Y$ such that%
\[
j_{p}\circ i_{X}(X)\subset X_{p}\text{ and }\widehat{u}\circ j_{p}\circ
i_{X}=u,
\]
where $i_{X}:X\rightarrow C\left(  B_{X^{\ast}}\right)  $ and $j_{p}:C\left(
B_{X^{\ast}}\right)  \rightarrow L_{p}\left(  \mu\right)  $ are the canonical
inclusions. Moreover, every absolutely $p$-summing linear operator is weakly compact.
\end{itemize}


From now on $p\in\left[  1,\infty\right)  $ and $X,X_{1},...,X_{n},Y$ are
Banach spaces over the same scalar field $\mathbb{K}=\mathbb{R}$ or
$\mathbb{C}.$ A continuous $n$-linear operator $T:X_{1}\times\cdots\times
X_{n}\rightarrow Y$ is \textit{$p$-dominated} if there is a constant $C\geq0$
such that
\[
\left(  \sum\limits_{j=1}^{m}\parallel T(x_{j}^{1},...,x_{j}^{n}%
)\parallel^{\frac{p}{n}}\right)  ^{n/p}\leq C\left(  \sup_{\varphi\in
B_{X_{1}^{\ast}}}\sum\limits_{j=1}^{m}\left\vert \varphi(x_{j}^{1})\right\vert
^{p}\right)  ^{1/p}\cdots\left(  \sup_{\varphi\in B_{X_{n}^{\ast}}}%
\sum\limits_{j=1}^{m}\left\vert \varphi(x_{j}^{n})\right\vert ^{p}\right)
^{1/p}%
\]
for all $x_{j}^{k}\in X_{j},$ all $m\in\mathbb{N}$ and $\left(  j,k\right)
\in\left\{  1,...,m\right\}  \times\left\{  1,...,n\right\}  $. This concept
is essentially due to Pietsch (see \cite{AlencarMatos, anais}) and lifts
several important properties of the original linear ideal of absolutely
summing operators to the multilinear framework. The terminology
\textquotedblleft$p$-dominated\textquotedblright, coined by M.C. Matos, is
motivated by the following Pietsch-Domination type theorem:

\begin{theorem}
[Pietsch, Geiss, 1985](\cite{geisse}) A continuous $n$-linear operator
$T:X_{1}\times\cdots\times X_{n}\rightarrow Y$ is $p$-dominated if and only if
there exist $C\geq0$ and regular probability measures $\mu_{j}$ on the Borel
$\sigma$-algebras of $B_{X_{j}^{\ast}}$ endowed with the weak star topologies
such that
\[
\left\Vert T\left(  x_{1},...,x_{n}\right)  \right\Vert \leq C\prod
\limits_{j=1}^{n}\left(  \int_{B_{X_{j}^{\ast}}}\left\vert \varphi\left(
x_{j}\right)  \right\vert ^{p}d\mu_{j}\left(  \varphi\right)  \right)  ^{1/p}%
\]
for every $x_{j}\in X_{j}$ and $j=1,...,n$.
\end{theorem}

\begin{corollary}
If $1\leq p\leq q<\infty$, then every $p$-dominated multilinear operator is
$q$-dominated$.$
\end{corollary}

The notion of $p$-semi-integral operator is another possible multilinear
generalization of the class of absolutely summing linear operators. If
$p\geq1,$ a continuous $n$-linear operator $T:X_{1}\times\cdots\times
X_{n}\rightarrow Y$ is \textit{$p$-semi-integral} if there exists a $C\geq0$
such that%
\[
\left(  \sum\limits_{j=1}^{m}\parallel T(x_{j}^{1},...,x_{j}^{n})\parallel
^{p}\right)  ^{1/p}\leq C\left(  \sup_{\left(  \varphi_{1},..,\varphi_{_{n}%
}\right)  \in B_{X_{1}^{\ast}}\times\cdots\times B_{X_{n}^{\ast}}}%
\sum\limits_{j=1}^{m}\mid\varphi_{1}(x_{j}^{1})...\varphi_{n}(x_{j}^{n}%
)\mid^{p}\right)  ^{1/p}%
\]
for every $m\in\mathbb{N}$, $x_{j}^{k}\in X_{k}$ with $k=1,...,n$ and
$j=1,...,m.$

This class dates back to the research report \cite{AlencarMatos} of R. Alencar
and M.C. Matos. As in the case of $p$-dominated multilinear operators, a
Pietsch Domination theorem is valid in this context:

\begin{theorem}
A continuous $n$-linear operator $T:X_{1}\times\cdots\times X_{n}\rightarrow
Y$ is $p$-semi-integral if and only if there exist $C\geq0$ and a regular
probability measure $\mu$ on the Borel $\sigma-$algebra $\mathcal{B}%
(B_{X_{1}^{^{\ast}}}\times\cdots\times$ $B_{X_{n}^{^{\ast}}})$ of
$B_{X_{1}^{^{\ast}}}\times\cdots\times$ $B_{X_{n}^{^{\ast}}}$ endowed with the
product of the weak star topologies $\sigma(X_{l}^{\ast},X_{l}),$ $l=1,...,n,$
such that
\[
\parallel T(x_{1},...,x_{n})\parallel\leq C\left(  \int_{B_{X_{1}^{\ast}%
}\times\cdots\times B_{X_{n}^{\ast}}}\mid\varphi_{1}(x_{1})...\varphi
_{n}(x_{n})\mid^{p}d\mu(\varphi_{1},...,\varphi_{n})\right)  ^{1/p}%
\]
for all $x_{j}\in X_{j}$, $j=1,...,n.$
\end{theorem}

\begin{corollary}
If $1\leq p\leq q<\infty,$ every $p$-semi-integral multilinear operator is $q$-semi-integral.
\end{corollary}

This class is strongly connected to the class of $p$-dominated multilinear
operators. For example, in \cite{CD} it is shown that every $p$-semi integral
$n$-linear operator is $np$-dominated.

The following result shows that we cannot expect to lift coincidence results
of the linear case to dominated multilinear operators:

\begin{theorem}
[Jarchow, Palazuelos, P\'{e}rez-Garc\'{\i}a and Villanueva, 2007]%
\label{ffy}(\cite{Jar}) For every $n\geq3$ and every $p\geq1$ and every
infinite dimensional Banach space $X$ there exists a continuous $n$-linear
operator $T:X\times\cdots\times X\rightarrow\mathbb{K}$ that fails to be $p$-dominated.
\end{theorem}

Since $p$-semi-integral $n$-linear operators are $np$-dominated, we have:

\begin{corollary}
For every $n\geq3$, every $p\geq1$ and every infinite dimensional Banach space
$X$ there exists a continuous $n$-linear operator $T:X\times\cdots\times
X\rightarrow\mathbb{K}$ that fails to be $p$-semi-integral.
\end{corollary}

So, in view of the previous result, it is obvious that we cannot expect a
Grothendieck type theorem for dominated or semi-integral operators. In this
direction, the classes of multiple summing multilinear operators
(\cite{Bombal, collec}), strongly multiple summing multilinear operators
(\cite{port}) and strongly summing multilinear operators (\cite{dimant}) are
other possible generalizations, with a quite better behavior if we are
interested in lifting coincidence theorems, like Grothendieck's theorem. But,
as a matter of fact, none of these classes lifts all the main properties of
absolutely summing linear operators to the multilinear setting.

In \cite{RuSaFact}, a variant of the notion of $p$-dominated polynomials which
satisfy (in a very natural form) a Pietsch factorization type theorem, is
introduced. A continuous $n$-homogeneous polynomial $P:X\rightarrow Y$ is
\textit{factorable $p$-dominated} if there is a $C\geq0$ such that for every
$x_{j}^{i}\in X$, and scalars $\lambda_{j}^{i}$, $1\leq j\leq m_{1}$, $1\leq
i\leq m_{2}$ and all positive integers $m_{1},m_{2},$ we have%
\[
\left(
{\displaystyle\sum\limits_{j=1}^{m_{1}}}
\left\Vert
{\displaystyle\sum\limits_{i=1}^{m_{2}}}
\lambda_{j}^{i}P\left(  x_{j}^{i}\right)  \right\Vert ^{p}\right)  ^{\frac
{1}{p}}\leq C\sup_{\varphi\in B_{X^{\ast}}}\left(
{\displaystyle\sum\limits_{j=1}^{m_{1}}}
\left\vert
{\displaystyle\sum\limits_{i=1}^{m_{2}}}
\lambda_{j}^{i}\varphi\left(  x_{j}^{i}\right)  ^{n}\right\vert ^{p}\right)
^{\frac{1}{p}}.
\]
The natural multilinear version of the notion of \textquotedblleft factorable
$p$-dominated polynomials\textquotedblright\ seems to be:

\begin{definition}
A continuous $n$-linear operator $T:X_{1}\times\cdots\times X_{n}\rightarrow
Y$ is factorable $p$-dominated if there is a constant $C\geq0$ such that for
every $x_{k,j}^{i}\in X_{k}$, and scalars $\lambda_{j}^{i}$, $1\leq j\leq
m_{1}$, $1\leq i\leq m_{2}$ and all positive integers $m_{1},m_{2},$ we have%
\[
\left(
{\displaystyle\sum\limits_{j=1}^{m_{1}}}
\left\Vert
{\displaystyle\sum\limits_{i=1}^{m_{2}}}
\lambda_{j}^{i}T\left(  x_{1,j}^{i},...,x_{n,j}^{i}\right)  \right\Vert
^{p}\right)  ^{\frac{1}{p}}\leq C\sup_{\substack{\varphi_{k}\in B_{X_{k}%
^{\ast}}\\k=1,...,n}}\left(
{\displaystyle\sum\limits_{j=1}^{m_{1}}}
\left\vert
{\displaystyle\sum\limits_{i=1}^{m_{2}}}
\lambda_{j}^{i}\varphi_{1}\left(  x_{1,j}^{i}\right)  \cdots\varphi_{n}\left(
x_{n,j}^{i}\right)  \right\vert ^{p}\right)  ^{\frac{1}{p}}.
\]
\bigskip
\end{definition}

These notions have some connection with the idea of weighted summability,
sketched in \cite{mathz}. It is likely that this class has a nice
factorization theorem (like its polynomial version) but a simple calculation
shows that any factorable $p$-dominated multilinear operator is $p$%
-semi-integral and thus we have:

\begin{proposition}
For every $n\geq3$ and every $p\geq1$ and every infinite dimensional Banach
space $X$ there exists a continuous $n$-linear operator $T:X\times\cdots\times
X\rightarrow\mathbb{K}$ that fails to be factorable $p$-dominated. A fortiori,
regardless of the Banach space $Y$, there exists a continuous $n$-linear
operator $T:X\times\cdots\times X\rightarrow Y$ that fails to be factorable
$p$-dominated.
\end{proposition}

So, since we are looking for classes that also lift coincidence results to the
multilinear setting, the class of factorable $p$-dominated multilinear
operators is not what we are searching.

A continuous $n$-linear operator $T:X_{1}\times\cdots\times X_{n}\rightarrow
Y$ is \textit{strongly $p$-summing} if there exists a constant $C\geq0$ such
that
\begin{equation}
\left(  \sum\limits_{j=1}^{m}\parallel T(x_{j}^{1},...,x_{j}^{n})\parallel
^{p}\right)  ^{1/p}\leq C\left(  \underset{\phi\in B_{\mathcal{L}%
(X_{1},...,X_{n};\mathbb{K})}}{\sup}\sum\limits_{j=1}^{m}\mid\phi(x_{j}%
^{1},...,x_{j}^{n})\mid^{p}\right)  ^{1/p}. \label{in}%
\end{equation}
for every $m\in\mathbb{N}$, $x_{j}^{k}\in X_{k}$ with $k=1,...,n$ and
$j=1,...,m.$

The class of strongly $p$-summing multilinear operators is due to V. Dimant
\cite{dimant} and according to \cite{pp0, pp3} it is perhaps the class that
best translates to the multilinear setting the properties of the original
linear concept. For example, a Grothendieck type theorem and a
Pietsch-Domination type theorem are valid:

\begin{theorem}
[Grothendieck-type theorem](\cite{dimant}) Every continuous $n$-linear
operator $T:\ell_{1}\times\cdots\times\ell_{1}\rightarrow\ell_{2}$ is strongly
$1$-summing.
\end{theorem}

\begin{theorem}
[Pietsch Domination type theorem](\cite{dimant}) A continuous $n$-linear
operator $T:X_{1}\times\cdots\times X_{n}\rightarrow Y$ is strongly
$p$-summing if, and only if, there are a probability measure $\mu$ on
$B_{(X_{1}\widehat{\otimes}_{\pi}\cdots\widehat{\otimes}_{\pi}X_{n})^{\ast}}$,
with the weak-star topology, and a constant $C\geq0$ so that
\begin{equation}
\left\Vert T\left(  x_{1},...,x_{n}\right)  \right\Vert \leq C\left(
\int_{B_{(X_{1}\widehat{\otimes}_{\pi}\cdots\widehat{\otimes}_{\pi}%
X_{n})^{\ast}}}\left\vert \varphi\left(  x_{1}\otimes\cdots\otimes
x_{n}\right)  \right\vert ^{p}d\mu\left(  \varphi\right)  \right)  ^{\frac
{1}{p}} \label{7out08c}%
\end{equation}
for all $(x_{1},...,x_{n})\in X_{1}\times\cdots\times X_{n}.$
\end{theorem}

\begin{corollary}
If $p\leq q$ then every strongly $p$-summing multilinear operator is strongly
$q$-summing.
\end{corollary}

It is not hard to prove that a Dvoretzky-Rogers Theorem is also valid for this class:

\begin{theorem}
[Dvoretzky-Rogers type theorem]Every continuous $n$-linear operator
$T:X\times\cdots\times X\rightarrow X$ is strongly $p$-summing if, and only
if, $\dim X<\infty.$
\end{theorem}

A property fulfilled by the class of absolutely summing operators which is not
lifted to the multilinear framework by the notion of strong summability is the
weak compactness. In fact, it is well known that every absolutely $p$-summing
linear operator is weakly compact, but Carando and Dimant have shown that
there exist strongly $p$-summing multilinear operators that fail to be weakly
compact \cite{na}. This result implies that a natural version of the Pietsch
Factorization Theorem is not valid for strongly summing multilinear operators,
as we will see below.

Suppose that the following factorization theorem holds: $T:X_{1}\times
\cdots\times X_{n}\rightarrow Y$ is strongly $p$-summing if and only if there
is a regular Borel probability measure $\mu$ on $B_{(X_{1}\widehat{\otimes
}_{\pi}\cdots\widehat{\otimes}_{\pi}X_{n})^{\ast}}$, with the weak-star
topology, a closed subspace $Z_{p}$ of $L_{p}\left(  B_{(X_{1}\widehat
{\otimes}_{\pi}\cdots\widehat{\otimes}_{\pi}X_{n})^{\ast}},\mu\right)  $ and a
continuous linear operator $\widehat{T}:Z_{p}\rightarrow Y$ such that%
\[
j_{p}\circ i_{X_{1}\times\cdots\times X_{n}}(X_{1}\times\cdots\times
X_{n})\subset Z_{p}\text{ and }\widehat{T}\circ j_{p}\circ i_{X_{1}%
\times\cdots\times X_{n}}=T,
\]
where%
\[
i_{X_{1}\times\cdots\times X_{n}}:X_{1}\times\cdots\times X_{n}\rightarrow
C\left(  B_{(X_{1}\widehat{\otimes}_{\pi}\cdots\widehat{\otimes}_{\pi}%
X_{n})^{\ast}}\right)
\]
is the canonical $n$-linear map $i_{X_{1}\times\cdots\times X_{n}}%
(x_{1},...,x_{n})\left(  \varphi\right)  =\varphi(x_{1}\otimes\cdots\otimes
x_{n})$ and%
\[
j_{p}:C\left(  B_{(X_{1}\widehat{\otimes}_{\pi}\cdots\widehat{\otimes}_{\pi
}X_{n})^{\ast}}\right)  \rightarrow L_{p}\left(  B_{(X_{1}\widehat{\otimes
}_{\pi}\cdots\widehat{\otimes}_{\pi}X_{n})^{\ast}},\mu\right)
\]
is the canonical linear inclusion.

Since $j_{p}$ is absolutely $p$-summing (and thus weakly compact), then we
conclude that the set $j_{p}\left(  i_{X_{1}\times\cdots\times X_{n}}%
(B_{X_{1}}\times\cdots\times B_{X_{n}})\right)  $ is relatively weakly compact
in $Z_{p}.$ Since $\widehat{T}$ is continuous and linear, then $T(B_{X_{1}%
},...,B_{X_{n}})=\widehat{T}\left(  j_{p}\left(  i_{X_{1}\times\cdots\times
X_{n}}(B_{X_{1}}\times\cdots\times B_{X_{n}})\right)  \right)  $ is relatively
weakly compact in $Y$ and thus $T$ is weakly compact, but this is not true in
general (\cite{na}).

In this paper we combine the idea of factorable summability from
\cite{RuSaFact} with the notion of strongly $p$-summing multilinear operators
and we show that the new class we introduce recovers all these lacks suffered
by the former multilinear extensions.

\section{Factorable strongly $p$-summing multilinear operators}

\label{multilinear}

The following definition is inspired in ideas from \cite{RuSaFact}, adapted to
the notion of strongly summing multilinear operators:

\begin{definition}
A continuous $n$-linear operator $T:X_{1}\times\cdots\times X_{n}\rightarrow
Y$ is \textit{factorable strongly $p$-summing} if there is a constant $C\geq0$
such that for every $x_{k,j}^{i}\in X_{k}$, and scalars $\lambda_{k,j}^{i}$,
$1\leq j\leq m_{1}$, $1\leq i\leq m_{2}$ and all positive integers
$m_{1},m_{2},$ we have%
\[
\left(
{\displaystyle\sum\limits_{j=1}^{m_{1}}}
\left\Vert
{\displaystyle\sum\limits_{i=1}^{m_{2}}}
\lambda_{j}^{i}T\left(  x_{1,j}^{i},...,x_{n,j}^{i}\right)  \right\Vert
^{p}\right)  ^{\frac{1}{p}}\leq C\sup_{\left\Vert \varphi\right\Vert \leq
1}\left(
{\displaystyle\sum\limits_{j=1}^{m_{1}}}
\left\vert
{\displaystyle\sum\limits_{i=1}^{m_{2}}}
\lambda_{j}^{i}\varphi\left(  x_{1,j}^{i},...,x_{n,j}^{i}\right)  \right\vert
^{p}\right)  ^{\frac{1}{p}}.
\]
where the supremum is taken over all the continuous $n$-linear functionals
$\varphi:X_{1}\times\cdots\times X_{n}\rightarrow\mathbb{K}$ of norm less or
equal than $1$. The class of all factorable strongly $p$-summing $n$-linear
operators $T:X_{1}\times\cdots\times X_{n}\rightarrow Y$ is denoted by
$\Pi_{FSt,p}(X_{1},\ldots,X_{n};Y)$ and endowed with the norm $\Vert\cdot
\Vert_{FSt,p}$, where $\Vert T\Vert_{FSt,p}$ is given by the infimum of all
constant $C$ fulfilling the above inequality.
\end{definition}

Note that if $T$ is factorable strongly $p$-summing then making $m_{2}=1$ and
$\lambda_{j}^{1}=1$ for all $j=1,\ldots,m_{1}$, we have%
\[
\left\Vert \left(  T\left(  x_{1,j}^{1},...,x_{n,j}^{1}\right)  \right)
_{j=1}^{m_{1}}\right\Vert _{p}\leq C\sup_{\left\Vert \varphi\right\Vert \leq
1}\left(
{\displaystyle\sum\limits_{j=1}^{m_{1}}}
\left\vert \varphi\left(  x_{1,j}^{1},...,x_{n,j}^{1}\right)  \right\vert
^{p}\right)  ^{\frac{1}{p}},
\]
i.e., $T$ is strongly $p$-summing. In particular, whenever $n=1$, $\Pi
_{FSt,p}(X_{1};Y)=\Pi_{p}(X_{1};Y)$ is the class of all absolutely $p$-summing
operators from $X_{1}$ to $Y$.

The ideal property is straightforward. It is also trivial that every
scalar-valued $n$-linear operator is factorable strongly $p$-summing.
Straightforward calculations show that this class forms a Banach multi-ideal.


As we will see in Section \ref{wealth}, this class preserves the nice
properties of the class of strongly summing multilinear operators and has
extra desirable properties: weak compactness and a factorization theorem.

\begin{theorem}
[Pietsch-Domination type theorem]\label{p0} A continuous $n$-linear operator
$T:X_{1}\times\cdots\times X_{m}\rightarrow Y$ is factorable strongly
$p$-summing if and only if there is a regular probability measure $\mu$ on
$B_{(X_{1}\widehat{\otimes}_{\pi}\cdots\widehat{\otimes}_{\pi}X_{n})^{\ast}},$
endowed with the weak-star topology, and a constant $C\geq0$, such that%
\[
\left\Vert
{\displaystyle\sum\limits_{i=1}^{m}}
\lambda^{i}T\left(  x_{1}^{i},...,x_{n}^{i}\right)  \right\Vert \leq C\left(
{\displaystyle\int\nolimits_{B_{(X_{1}\widehat{\otimes}_{\pi}\cdots
\widehat{\otimes}_{\pi}X_{n})^{\ast}}}}
\left\vert
{\displaystyle\sum\limits_{i=1}^{m}}
\lambda^{i}\varphi\left(  x_{1}^{i},...,x_{n}^{i}\right)  \right\vert ^{p}%
d\mu\left(  \varphi\right)  \right)  ^{\frac{1}{p}}.
\]

\end{theorem}

\begin{proof}
The notion of factorable strongly $p$-summing multilinear operator is
precisely the concept of $RS$-abstract $p$-summing (see \cite{pp1, pp2, pp3})
for%
\[
R:B_{(X_{1}\widehat{\otimes}_{\pi}\cdots\widehat{\otimes}_{\pi}X_{n})^{\ast}%
}\times\left(  \mathbb{K}\times X_{1}\times\cdots\times X_{n}\right)
^{\mathbb{N}}\times\left\{  0\right\}  \rightarrow\left[  0,\infty\right)
\]
given by%
\[
R\left(  \varphi,\left(  \lambda^{1},x_{1}^{1},...,x_{n}^{1}\right)
,...,\left(  \lambda^{m},x_{1}^{m},...,x_{n}^{m}\right)  ,0\right)
=\left\vert
{\displaystyle\sum\limits_{i=1}^{m}}
\lambda^{i}\varphi\left(  x_{1}^{i}\otimes\cdots\otimes x_{n}^{i}\right)
\right\vert
\]
and%
\[
S:\mathcal{L}\left(  X_{1},...,X_{n};Y\right)  \times\left(  \mathbb{K}\times
X_{1}\times\cdots\times X_{n}\right)  ^{\mathbb{N}}\times\left\{  0\right\}
\rightarrow\left[  0,\infty\right)
\]
given by%
\[
S\left(  T,\left(  \lambda^{1},x_{1}^{1},...,x_{n}^{1}\right)  ,...,\left(
\lambda^{m},x_{1}^{m},...,x_{n}^{m}\right)  ,0\right)  =\left\Vert
{\displaystyle\sum\limits_{i=1}^{m}}
\lambda^{i}T\left(  x_{1}^{i},...,x_{n}^{i}\right)  \right\Vert .
\]
Since $R$ and $S$ satisfy the hypotheses of the general Pietsch Domination
Theorem, the result follows straightforwardly.
\end{proof}

\begin{theorem}
[Pietsch-Factorization type theorem]A continuous $n$-linear operator
$T:X_{1}\times\cdots\times X_{n}\rightarrow Y$ is factorable strongly
$p$-summing if and only if there exist a regular probability measure $\mu$ on
$B_{(X_{1}\widehat{\otimes}_{\pi}\cdots\widehat{\otimes}_{\pi}X_{n})^{\ast}},$
endowed with the weak-star topology, a constant $C\geq0$, a closed subspace
$Z_{p}$ of $L_{p}\left(  B_{(X_{1}\widehat{\otimes}_{\pi}\cdots\widehat
{\otimes}_{\pi}X_{n})^{\ast}},\mu\right)  $ and a continuous linear operator
$\widehat{T}:Z_{p}\rightarrow Y$ such that%
\[
j_{p}\circ i_{X_{1}\times\cdots\times X_{n}}(X_{1}\times\cdots\times
X_{n})\subset Z_{p}\text{ and }\widehat{T}\circ j_{p}\circ i_{X_{1}%
\times\cdots\times X_{n}}=T,
\]
where%
\[
i_{X_{1}\times\cdots\times X_{n}}:X_{1}\times\cdots\times X_{m}\rightarrow
C\left(  B_{(X_{1}\widehat{\otimes}_{\pi}\cdots\widehat{\otimes}_{\pi}%
X_{n})^{\ast}}\right)
\]
is the canonical $n$-linear map $i_{X_{1}\times\cdots\times X_{n}}%
(x_{1},...,x_{n})\left(  \varphi\right)  =\varphi(x_{1}\otimes\cdots\otimes
x_{n})$ and%
\[
j_{p}:C\left(  B_{(X_{1}\widehat{\otimes}_{\pi}\cdots\widehat{\otimes}_{\pi
}X_{n})^{\ast}}\right)  \rightarrow L_{p}\left(  B_{(X_{1}\widehat{\otimes
}_{\pi}\cdots\widehat{\otimes}_{\pi}X_{n})^{\ast}},\mu\right)
\]
is the canonical linear inclusion.
\end{theorem}

\begin{proof}
Suppose that $T$ is factorable strongly $p$-summing. Let $\mu$ be the measure
given by the Pietsch Domination Theorem (Theorem \ref{p0}) applied to $T$. Let
$W_{p}$ be the subspace of $L_{p}\left(  B_{(X_{1}\widehat{\otimes}_{\pi
}\cdots\widehat{\otimes}_{\pi}X_{n})^{\ast}},\mu\right)  $ given by the linear
span of $j_{p}\circ i_{X_{1}\times\cdots\times X_{m}}(X_{1}\times\cdots\times
X_{m}).$ Define the linear operator $\widehat{T}:W_{p}\rightarrow Y$ by%
\[
\widehat{T}\left(  z\right)  =%
{\displaystyle\sum\limits_{i=1}^{n}}
\lambda_{i}T\left(  x_{1}^{i},...,x_{n}^{i}\right)
\]
for%
\[
z=%
{\displaystyle\sum\limits_{i=1}^{n}}
\lambda_{i}\langle\cdot,\left(  x_{1}^{i}\otimes\cdots\otimes x_{n}%
^{i}\right)  \rangle\in W_{p}.
\]
Note that $\widehat{T}$ is well-defined. In fact, if
\[
z_{1}=%
{\displaystyle\sum\limits_{i=1}^{m_{1}}}
\lambda_{i}\langle\cdot,\left(  x_{1}^{i}\otimes\cdots\otimes x_{n}%
^{i}\right)  \rangle\text{ and }z_{2}=%
{\displaystyle\sum\limits_{i=1}^{m_{2}}}
\alpha_{i}\langle\cdot,\left(  y_{1}^{i}\otimes\cdots\otimes y_{n}^{i}\right)
\rangle
\]
coincide in $W_{p}$, then considering
\[
w:=%
{\displaystyle\sum\limits_{i=1}^{m_{1}}}
\lambda_{i}\langle\cdot,\left(  x_{1}^{i}\otimes\cdots\otimes x_{n}%
^{i}\right)  \rangle-%
{\displaystyle\sum\limits_{i=1}^{m_{2}}}
\alpha_{i}\langle\cdot,\left(  y_{1}^{i}\otimes\cdots\otimes y_{n}^{i}\right)
\rangle,
\]
we have $w=0$ almost everywhere in $W_{p}$, i.e.,%
\[%
{\displaystyle\int\nolimits_{B_{(X_{1}\widehat{\otimes}_{\pi}\cdots
\widehat{\otimes}_{\pi}X_{n})^{\ast}}}}
\left\vert
{\displaystyle\sum\limits_{i=1}^{m_{1}}}
\lambda_{i}\varphi\left(  x_{1}^{i}\otimes\cdots\otimes x_{n}^{i}\right)  -%
{\displaystyle\sum\limits_{i=1}^{m_{2}}}
\alpha_{i}\varphi\left(  y_{1}^{i}\otimes\cdots\otimes y_{n}^{i}\right)
\right\vert ^{p}d\mu\left(  \varphi\right)  =0.
\]
Thus, from the domination theorem,%
\begin{align*}
&  \left\Vert
{\displaystyle\sum\limits_{i=1}^{m_{1}}}
\lambda_{i}T\left(  x_{1}^{i},...,x_{n}^{i}\right)  -%
{\displaystyle\sum\limits_{i=1}^{m_{2}}}
\alpha_{i}T\left(  y_{1}^{i},...,y_{n}^{i}\right)  \right\Vert \\
&  \leq C\left(
{\displaystyle\int\nolimits_{B_{(X_{1}\widehat{\otimes}_{\pi}\cdots
\widehat{\otimes}_{\pi}X_{n})^{\ast}}}}
\left\vert
{\displaystyle\sum\limits_{i=1}^{m_{1}}}
\lambda_{i}\varphi\left(  x_{1}^{i}\otimes\cdots\otimes x_{n}^{i}\right)  -%
{\displaystyle\sum\limits_{i=1}^{m_{2}}}
\alpha_{i}\varphi\left(  y_{1}^{i}\otimes\cdots\otimes y_{n}^{i}\right)
\right\vert ^{p}d\mu\left(  \varphi\right)  \right)  ^{\frac{1}{p}}=0
\end{align*}
and we conclude that
\[
\widehat{T}(z_{1})-\widehat{T}(z_{2})=%
{\displaystyle\sum\limits_{i=1}^{m_{1}}}
\lambda_{i}T\left(  x_{1}^{i},...,x_{n}^{i}\right)  -%
{\displaystyle\sum\limits_{i=1}^{m_{2}}}
\alpha_{i}T\left(  y_{1}^{i},...,y_{n}^{i}\right)  =0.
\]
Note also that for $z=%
{\displaystyle\sum\limits_{i=1}^{m}}
\lambda_{i}\langle\cdot,\left(  x_{1}^{i}\otimes\cdots\otimes x_{n}%
^{i}\right)  \rangle\in W_{p}$ we have%
\begin{align*}
\left\Vert \widehat{T}\left(  z\right)  \right\Vert  &  =\left\Vert
{\displaystyle\sum\limits_{i=1}^{m}}
\lambda_{i}T\left(  x_{1}^{i},...,x_{n}^{i}\right)  \right\Vert \\
&  \leq C\left(
{\displaystyle\int\nolimits_{B_{(X_{1}\widehat{\otimes}_{\pi}\cdots
\widehat{\otimes}_{\pi}X_{n})^{\ast}}}}
\left\vert
{\displaystyle\sum\limits_{i=1}^{m}}
\lambda_{i}\varphi\left(  x_{1}^{i}\otimes\cdots\otimes x_{n}^{i}\right)
\right\vert ^{p}d\mu\left(  \varphi\right)  \right)  ^{\frac{1}{p}}\\
&  =C\left\Vert z\right\Vert _{L_{p}\left(  \mu\right)  }%
\end{align*}
and $\widehat{T}$ is continuous. It is obvious that from the very definition
of $\widehat{T}$ we have $\widehat{T}\circ j_{p}\circ i_{X_{1}\times
\cdots\times X_{n}}=T$. Now we extend $\widehat{T}$ to $Z_{p}=\overline{W_{p}%
}.$ The converse is immediate.
\end{proof}

\section{Factorable strongly $p$-summing polynomials}

\label{polynomials}

The $m$-fold symmetric tensor product of $X$ is the linear span of all tensors
of the form $x\otimes\cdots\otimes x$, $x\in X$, and is denoted by
$\otimes^{m,s}X$. This space is endowed with the $s$-projective tensor norm,
defined as
\[
\pi_{s}(z)=\inf\{\sum_{j=1}^{k}|\lambda_{j}|\Vert x_{j}\Vert^{n}%
:k\in\mathbb{N},z=\sum_{j=1}^{k}\lambda_{j}x_{j}\otimes\cdots\otimes x_{j}\},
\]
for $z\in\otimes_{m,s}X$. Let $\hat{\otimes}_{\pi_{s}}^{m,s}X$ denote the
completion of $\otimes_{\pi_{s}}^{m,s}X$.

Given $P\in{\mathcal{P}}(^{m}X;Y)$, the linearization of $P$ is the unique
linear operator $P_{L,s}:\hat\otimes^{m,s}_{\pi_{s}} X\to Y$ such that
$P_{L,s}(x\otimes\cdots\otimes x)=P(x)$ for all $x\in X$. Ryan \cite{Ryan}
proved that the correspondence $P\leftrightarrow P_{L,s}$ establishes a
isometric isomorphism between the space ${\mathcal{P}}(^{m}X)$, endowed with
the usual sup norm, and the strong dual of $\hat\otimes^{m,s}_{\pi_{s}} X$.
Another map associated to $P\in{\mathcal{P}}(^{m}X;Y)$ is the unique
continuous symmetric $m$-linear mapping $\check P$ that satisfies $\check
P(x,\ldots, x)=P(x)$, for all $x\in X$. It is well known that $\|\check
P\|\leq c(m,X)\|P\|$ for all $P\in{\mathcal{P}}(^{m}X)$, where $c(m,X)$ is the
$m$-th polarization constant of $X$. For the general theory of homogeneous
polynomials we refer to \cite{dineen} and \cite{mujica}.

Concomitantly to multilinear mappings, factorable strongly $p$-summing
homogeneous polynomials can be introduced. Our aim is to prove that both
classes coincide in the sense that a polynomial is factorable strongly
$p$-summing if and only if its associated symmetric multilinear mapping is
factorable strongly $p$-summing. Moreover, we will see the deep relationship
between factorable strong summability and absolute summability by proving
that, for an homogeneous polynomial, it is equivalent that the polynomial is
factorable strongly summing to that its linearization is an absolutely summing
operator. To attain this purpose, we will show that both, factorable strongly
$p$-summing polynomials and factorable strongly $p$-summing multilinear
operators, form composition ideals.

\begin{definition}
A continuous $n$-homogeneous polynomial $P:X\rightarrow Y$ is
\textit{factorable strongly $p$-summing} if there is a $C\geq0$ such that for
every $x_{j}^{i}\in X$, and scalars $\lambda_{j}^{i}$, $1\leq j\leq m_{1}$,
$1\leq i\leq m_{2}$ and all positive integers $m_{1},m_{2},$ we have that
\[
\|(\sum_{i=1}^{m_{2}} \lambda_{j}^{i} P(x_{j}^{i}))_{j}\|_{p}\leq C
\sup_{\|q\|\leq1, q\in{\mathcal{P}}(^{m}X)}(\sum_{j=1}^{m_{1}}|\sum
_{i=1}^{m_{2}}\lambda_{j}^{i} q(x_{j}^{i})|^{p})^{1/p}.
\]
The class of all factorable strongly $p$-summing $m$-homogeneous polynomials
from $X$ to $Y$ is denoted by ${\mathcal{P}}_{FSt,p}(^{m}X;Y)$ and endowed
with the norm $\|\cdot\|_{FSt,p}$ given by the infimum of all constants $C$
fulfilling the above inequality.
\end{definition}

It is clear that factorable $p$-dominated polynomials are factorable strongly
$p$-summing. An easy calculation shows the following ideal property:

\begin{proposition}
\label{polid} If $P\in{\mathcal{P}}_{FSt,p}(^{m}X;Y)$ and $u:G\to X$, $v:Y\to
Z$ are continuous linear operators then $v\circ P\circ u\in{\mathcal{P}%
}_{FSt,p}(^{m}G;Z)$ and $\|v\circ P\circ u\|_{FSt,p}\leq\| v\| \cdot
\|P\|_{FSt,p}\|u\|^{m}$.
\end{proposition}

It is not difficult to complete Proposition \ref{polid} and show that
factorable strongly $n$-homogeneous polynomials form an ideal of polynomials
(for the definition of ideal of polynomials we refer to \cite{BBJP}).

Dimant \cite{dimant} introduced the class of \textit{strongly $p$-summing}
$m$-homogeneous polynomials from $X$ to $Y$ as those $m$-homogeneous
polynomials $P:X\to Y$ that satisfy that there exists $K>0$ such that for any
$n\in\mathbb{N}$ and any $x_{1},\ldots, x_{n}\in X$,
\[
(\sum_{j=1}^{n}\|P(x_{j})\|^{p})^{1/p}\leq K \sup_{\|q\|\leq1, q\in
{\mathcal{P}}(^{m}X)} (\sum_{j=1}^{n} |q(x_{j}))|^{p})^{1/p}.
\]
In \cite[Proposition 3.2]{dimant} it is proved that if the linearization
$P_{L,s}$ of $P\in{\mathcal{P}}(^{m}X;Y)$ is absolutely $p$-summing then $p$
is strongly $p$-summing. However, the converse is not true (see \cite[Example
3.3]{na}). The reason, as for $p$-dominated polynomials, is that not every
strongly $p$-summing polynomial is weakly compact. So, once again, the lack of
connection with weak compactness turns out to be a deep inconvenience in the
way that strongly $p$-summing polynomials generalize absolutely $p$-summing
linear operators. Even if a domination holds also for strongly $p$-summing
polynomials \cite[Proposition 3.2]{dimant}, no factorization theorem is
expected. Let us prove a factorization theorem for factorable strongly
$p$-summing polynomials. We first need a domination theorem, that is obtained
as a particular case of \cite[Theorem~2.2]{pp1}. We denote by $\delta:X\to
C(B_{{\mathcal{P}}(^{m}X)})$ the $m$-homogeneous polynomial given by
$\delta(x):=\delta_{x}:B_{{\mathcal{P}}(^{m}X)}\to\mathbb{K}$, where
$\delta_{x}(P):=P(x)$. Considering that the space of continuous $m$%
-homogeneous polynomials is a dual space (see \cite{Ryan}), its closed unit
ball $B_{{\mathcal{P}}(^{m}X)}$ is a weak-star compact set.

\begin{theorem}
[Pietsch-Domination type theorem]\label{dominationpol} Let $P\in{\mathcal{P}%
}(^{m}X;Y)$. Then $P$ is factorable strongly $p$-summing if and only if there
exists a regular Borel probability measure $\mu$ on $B_{\mathcal{P}(^{m}X)}$,
endowed with the weak-star topology, such that
\[
\| \sum_{i=1}^{k} \lambda^{i} P(x^{i})\|\leq C(\int_{B_{{\mathcal{P}}(^{m}X)}}
|\sum_{i=1}^{k} \lambda^{i} q(x^{i})|^{p} \, d\mu)^{1/p}
\]
for all $x^{1},\ldots, x^{k} \in X$ and $\lambda^{1},\ldots, \lambda^{k}%
\in\mathbb{K}$.
\end{theorem}

\begin{proof}
It is a particular case of \cite[Theorem~2.2]{pp1} analogous to the proof of
Theorem \ref{p0}.
\end{proof}

We shall need the following result to prove the sufficiency of the
Factorization Theorem. Besides, Proposition~\ref{comp} will be the key for our
purposes to obtain that factorable strongly $p$-summing homogeneous
polynomials form a composition ideal.

\begin{proposition}
\label{comp} If $Q\in{\mathcal{P}}(^{m}G;X)$ and $u:X\to Y$ is an absolutely
$p$-summing linear operator, then $u\circ Q\in{\mathcal{P}}_{FSt,p}(^{m}G;Y)$
and $\|u\circ Q\|_{FSt,p}\leq\pi_{p}(u)\|Q\|$.
\end{proposition}

\begin{proof}
Let $m_{1},m_{2}$ be positive integers, $x_{j}^{i}\in X$, and scalars
$\lambda_{j}^{i}$, $1\leq j\leq m_{1}$, $1\leq i\leq m_{2}$. Then,
\begin{align*}
(\sum_{j=1}^{m_{1}} \| (\sum_{i=1}^{m_{2}} \lambda_{j}^{i} u\circ Q(x_{j}%
^{i})\|^{p})^{1/p}  &  = (\sum_{j=1}^{m_{1}} \| u(\sum_{i=1}^{m_{2}}
\lambda_{j}^{i} Q(x_{j}^{i})\|^{p})^{1/p}\\
&  \leq\pi_{p}(u) \sup_{\|x^{*}\|\leq1, x^{*}\in X^{*}}(\sum_{j=1}^{m_{1}%
}|\langle x^{*}, \sum_{i=1}^{m_{2}}\lambda_{j}^{i} Q(x_{j}^{i})\rangle
|^{p})^{1/p}\\
&  \leq\pi_{p}(u) \|Q\| \sup_{\|x^{*}\|\leq1, x^{*}\in X^{*}}(\sum
_{j=1}^{m_{1}}| \sum_{i=1}^{m_{2}}\lambda_{j}^{i}\langle x^{*}, Q/\|Q\|(x_{j}%
^{i})\rangle|^{p})^{1/p}\\
&  \leq\pi_{p}(u) \|Q\| \sup_{\|q\|\leq1, q\in{\mathcal{P}}(^{m}G)}(\sum
_{j=1}^{m_{1}}|\sum_{i=1}^{m_{2}}\lambda_{j}^{i} q(x_{j}^{i})|^{p})^{1/p}.
\end{align*}

\end{proof}

\begin{theorem}
[Pietsch-Factorization type theorem]\label{factorizationpol} Let
$P\in{\mathcal{P}}(^{m}X;Y)$. Then $P$ is factorable strongly $p$-summing if
and only if there exists a regular Borel probability measure $\mu$ on
$B_{\mathcal{P}(^{m}X)}$, a closed subspace $G_{p}$ of $L_{p}(\mu)$ and a
continuous linear operator $v_{0}:G_{p}\rightarrow Y$ such that $j_{p}%
\circ\delta(X)\subset G_{p}$ and $v_{0}\circ j_{p}\circ\delta=P$, where
$j_{p}:C(B_{{\mathcal{P}}(^{m}X)})\rightarrow L_{p}(B_{{\mathcal{P}}(^{m}%
X)},\mu)$ is the canonical inclusion.

\end{theorem}

\begin{proof}
Assume first that $P$ is factorable strongly $p$-summing. Let $\mu$ be given
by Theorem~\ref{dominationpol}. Take $G_{p}$ the completion of the image by
$j_{p}$ of the linear span of $\delta(X)$. Define $v_{0}(j_{p} (\sum_{i=1}%
^{k}\lambda_{i}\delta_{x_{i}})):=\sum_{i=1}^{k}\lambda_{i}P(x_{i})$. To see
that $v_{0}$ is well defined, consider that $j_{p} (\sum_{i=1}^{k}\lambda
_{i}\delta_{x_{i}})=j_{p} (\sum_{i=1}^{l}\eta_{i}\delta_{y_{i}})$. Then
$w:=\sum_{i=1}^{k}\lambda_{i}\delta_{x_{i}}-\sum_{i=1}^{l}\eta_{i}%
\delta_{y_{i}}=0$ a.e. on $B_{{\mathcal{P}}(^{m}X)}$. Hence,
\[
\|\sum_{i=1}^{k}\lambda_{i}P(x_{i})-\sum_{i=1}^{l} \eta_{i}P(y_{i}%
)\|\leq\|P\|_{FSt,p}(\int_{B_{{\mathcal{P}}(^{m}X)}} |\sum_{i=1}^{k}%
\lambda_{i}q(x_{i})-\sum_{i=1}^{l} \eta_{i}q(y_{i})|^{p}\, d\mu)^{1/p}=0.
\]
Thus, $v_{0}(w)=0$. That proves that $v_{0}$ is well defined. The continuity
of $v_{0}$ follows from the calculations:
\begin{align*}
\| v_{0}(z)\|  &  = \|\sum_{i=1}^{k}\lambda_{i}P(x_{i})\|\leq\|P\|_{FSt,p}%
(\int_{B_{{\mathcal{P}}(^{m}X)}} |\sum_{i=1}^{k}\lambda_{i}q(x_{i})|^{p}
d\mu)^{1/p}\\
&  = \|P\|_{FSt,p}\| \sum_{i=1}^{k}\lambda_{i}\delta_{x_{i}}\|_{L_{p}(\mu
)}=\|P\|_{FSt,p}\|z\|_{L_{p}(\mu)}\\
\end{align*}
for any $z=j_{p}(\sum_{i=1}^{k}\lambda_{i}\delta_{x_{i}})$. The desired linear
operator is just the continuous extension of $v_{0}$ to $G_{p}$. The converse
follows from Proposition~\ref{comp}.
\end{proof}

\begin{corollary}
\label{idealcomp} Let $P\in{\mathcal{P}}(^{m}X;Y)$. Then $P\in{\mathcal{P}%
}_{FSt,p}(^{m}X;Y)$ if and only if $P=u\circ Q$, for some continuous
$m$-homogeneous polynomial $Q$ and some absolutely $p$-summing linear operator
$u$. In that case $\|P\|_{FSt,p}=\inf\{\pi_{p}(u)\|Q\|:P=u\circ Q\}$.
\end{corollary}

\begin{proof}
It follows from Theorem \ref{factorizationpol} and Proposition \ref{comp}.
\end{proof}

Corollary \ref{idealcomp} says that the ideal of all factorable strongly
$p$-summing $m$-homogeneous polynomials is the composition ideal with all
absolutely $p$-summing linear operators, that is, ${\mathcal{P}}_{FSt,p}%
=\Pi_{p}\circ{\mathcal{P}}$ (see \cite{BoPeRuPRIMS2007}). An analogous
argument for multilinear operators instead of polynomials yields to prove that
the ideal of all factorable strongly $p$-summing $m$-linear operators is the
composition ideal with all absolutely $p$-summing linear operators, that is,
$\Pi_{FSt,p}=\Pi_{p}\circ{\mathcal{L}}$.

\begin{remark}
In \cite{MeTo} it is shown an example of a continuous $m$-homogeneous
polynomial $P:X\to Y$ and $\phi\in\Pi_{as,r}(Y;Z)$ such that $\phi\circ P:X\to
Z$ is not $r$-dominated. By Proposition~\ref{comp}, $\phi\circ P$ is
factorable strongly $r$-summing. Therefore, the class of dominated polynomials
differs from the class of factorable strongly $r$-summing polynomials.
\end{remark}

\begin{theorem}
\label{linearization} Let $P\in{\mathcal{P}}(^{m}X;Y)$. The following are equivalent:

\begin{enumerate}
\item $P\in{\mathcal{P}}_{FSt,p}(^{m}X;Y).$

\item $P_{L,s}$ is absolutely $p$-summing.

\item $\check P\in\Pi_{FSt,p}(^{m}X;Y)$.
\end{enumerate}

In that case, $\|P\|_{FSt,p}=\pi_{p}(P_{L,s})$.
\end{theorem}

\begin{proof}
(1)$\Rightarrow$(2) Assume first that $P$ is factorable strongly $p$-summing.
Then
\begin{align*}
(\sum_{j=1}^{m_{1}}\Vert P_{L,s}(\sum_{i=1}^{m_{2}}\lambda_{i}^{j}x_{i}%
^{j}\otimes\cdots\otimes x_{i}^{j})\Vert^{p})^{1/p}  &  =(\sum_{j=1}^{m_{1}%
}\Vert\sum_{i=1}^{m_{2}}\lambda_{i}^{j}P_{L,s}(x_{i}^{j}\otimes\cdots\otimes
x_{i}^{j})\Vert^{p})^{1/p}\\
&  =(\sum_{j=1}^{m_{1}}\Vert\sum_{i=1}^{m_{2}}\lambda_{j}^{i}P(x_{i}^{j}%
)\Vert^{p})^{1/p}\\
&  \leq\Vert P\Vert_{FSt,p}\sup_{\Vert q\Vert\leq1,q\in{\mathcal{P}}(^{m}%
X)}(\sum_{j=1}^{m_{1}}|\sum_{i=1}^{m_{2}}\lambda_{j}^{i}q(x_{i}^{j}%
)|^{p})^{1/p}\\
&  =\Vert P\Vert_{FSt,p}\sup_{\Vert q\Vert\leq1,q\in(\hat{\otimes}_{\pi_{s}%
}^{m,s}X)^{\ast}}(\sum_{j=1}^{m_{1}}|\sum_{i=1}^{m_{2}}\lambda_{j}%
^{i}q(\otimes_{m}x_{i}^{j})|^{p})^{1/p}\\
&
\end{align*}
(2)$\Rightarrow$(1) follows from Proposition \ref{comp}.\newline%
(1)$\Leftrightarrow$(3) As ${\mathcal{P}}_{FSt,p}(^{m}X;Y)=\Pi_{p}%
\circ{\mathcal{P}}(^{m}X;Y)$, it follows from \cite[Proposition 3.2]%
{BoPeRuPRIMS2007} that $P\in{\mathcal{P}}_{FSt,p}(^{m}X;Y)$ if and only if
$\check{P}\in\Pi_{p}\circ{\mathcal{L}}(^{m}X;Y)$ and, similarly to Proposition
\ref{comp} it can be proved that $\Pi_{p}\circ{\mathcal{L}}(^{m}%
X;Y)=\Pi_{FSt,p}(^{m}X;Y)$.
\end{proof}

We finish this section with a Grothendieck type theorem:

\begin{theorem}
[Grothendieck type theorem]\textrm{ }If $m\geq1$ is a positive integer, then
${\mathcal{P}}(^{m}\ell_{1};\ell_{2})={\mathcal{P}}_{FSt,1}(^{m}\ell_{1}%
;\ell_{2}).$
\end{theorem}

\begin{proof}
Let $P\in{\mathcal{P}}(^{m}X;Y)$. Then $P_{L,s}\in{\mathcal{L}}(\hat{\otimes
}_{\pi_{s}}^{m,s}\ell_{1};\ell_{2})={\mathcal{L}}(\ell_{1};\ell_{2}%
)=\Pi_{as,1}(\ell_{1};\ell_{2})$. Theorem \ref{linearization} yields the result.
\end{proof}

\section{The wealth of factorable strong $p$-summability}

\label{wealth}

In this section it is shown that factorable strong $p$-summability is an
excellent non linear frame where linear results for absolute summability are
properly generalized to multilinear operators and polynomials. This evidences
the interest of this new class as it really reflects the good behavior of
absolute summability in the non linear context. Some of these results are
established for multilinear operators and some for homogeneous polynomials.
However, as a consequence of Theorem \ref{linearization} it is clear that one
can pass easily from one to each other.

The following results are consequences of Theorem \ref{linearization} and
their linear analogs (see \cite[Theorems 3.15 and 3.17]{DJT}):

\begin{proposition}
[Composition Theorem]If $u\in\Pi_{p}(X;Y)$ and $P\in{\mathcal{P}}_{FSt,q}%
(^{m}G;X)$ then $u\circ P\in{\mathcal{P}}_{FSt,r}(^{m}G;Y)$ for $1/r:=\min
\{1,1/p+1/q\}$.
\end{proposition}

\begin{proof}
By Theorem \ref{linearization} $P_{L,s}$ is absolutely $q$-summing. Then
$u\circ P_{L,s}$ is $r$-summing for $1/r:=\min\{1,1/p+1/q\}$ (see
\cite[Theorem 2.22]{DJT}). Since $u\circ P_{L,s}=(u\circ P)_{L,s}$, a second
application of Theorem \ref{linearization} yields the result.
\end{proof}

\begin{theorem}
[Extrapolation type theorem]Let $1<r<p<\infty,$ and let $X$ be a Banach space.
If ${\mathcal{P}}_{FSt,p}(^{m}X;\ell_{p})={\mathcal{P}}_{FSt,r}(^{m}X;\ell
_{p})$ then ${\mathcal{P}}_{FSt,p}(^{m}X;Y)={\mathcal{P}}_{FSt,1}(^{m}X;Y)$
for every Banach space $Y$.
\end{theorem}

Recall that given $1\leq p\leq \infty$ and $\lambda>1$, a Banach space $X$ is said to be an {\it ${\mathcal{L}}_{p,\lambda}$-space} if every finite dimensional subspace $E$ of $X$ is contained in a finite  dimensional subspace $F$ of $X$ for which there is an isomorphism $v:F\to \ell_p^{{\rm dim} F}$ with $\|v\|\cdot \|v^{-1}\|<\lambda$.

\begin{theorem}
[Lindenstrauss--Pe\l czy\'{n}ski type theorem]Let $1\leq p\leq2$ and
$2<q<\infty$. If $X$ is a Banach space and $Y$ is a subspace of an
${\mathcal{L}}_{p,\lambda}$-space, then ${\mathcal{P}}_{FSt,q}(^{m}%
X;Y)={\mathcal{P}}_{FSt,2}(^{m}X;Y)$.
\end{theorem}

We have already proved that Domination/Factorization Theorems are fulfilled in
the multilinear and polynomial classes of factorable strongly $p$-summing
maps. As a straightforward consequence of the Factorization Theorem
\ref{factorizationpol} we get

\begin{theorem}
Any factorable strongly $p$-summing polynomial is weakly compact.
\end{theorem}

An alternative way to prove it is the following: by
Theorem~\ref{linearization} the linearization of a factorable strongly
$p$-summing polynomial $P$ is absolutely $p$-summing and hence weakly compact.
By \cite{Ryan} this is equivalent to the weak compactness of $P$. The same
holds for the case of multilinear operators.

The Domination Theorem \ref{p0} also yields to the following inclusion theorem.

\begin{proposition}
[Inclusion Theorem]If $1\leq p\leq q<\infty$ then every factorable strongly
$p$-summing polynomial is factorable strongly $q$-summing.
\end{proposition}

The forthcoming lemmas \ref{sin}, \ref{870} and its consequences show that,
besides its good properties, the classes of factorable strongly $p$-summing
multilinear operators and polynomials have a coherent size.

\begin{lemma}
\label{sin} If every continuous $n$-linear operator $T:X_{1}\times\cdots\times
X_{n}\rightarrow Y$ is factorable strongly $p$-summing, then every continuous
linear operator $u_{j}:X_{j}\rightarrow Y$ is absolutely $p$-summing for every
$j=1,...,n$.
\end{lemma}

\begin{proof}
For the sake of simplicity, let us suppose $j=1.$ Let $u:X_{1}\rightarrow Y$
be a continuous linear operator and $\varphi_{j}\in X_{j}^{\ast}$, $j=2,...,n$
be non-null linear functionals. Then $T\left(  x_{1},...,x_{n}\right)
:=u\left(  x_{1}\right)  \varphi_{2}\left(  x_{2}\right)  ...\varphi
_{n}\left(  x_{n}\right)  $ is factorable strongly $p$-summing. Thus, in
particular, there is a $C>0$ such that
\[
\left(
{\displaystyle\sum\limits_{j=1}^{m_{1}}}
\left\Vert T\left(  x_{1,j},...,x_{n,j}\right)  \right\Vert ^{p}\right)
^{\frac{1}{p}}\leq C\sup_{\left\Vert \varphi\right\Vert \leq1}\left(
{\displaystyle\sum\limits_{j=1}^{m_{1}}}
\left\vert \varphi\left(  x_{1,j},...,x_{n,j}\right)  \right\vert ^{p}\right)
^{\frac{1}{p}}.
\]
Choose $a_{j}\in X_{j}$ such that $\varphi_{j}\left(  a_{j}\right)  =1$ for
all $j=2,...,n.$ Thus%
\[
\left(
{\displaystyle\sum\limits_{j=1}^{m_{1}}}
\left\Vert T\left(  x_{1,j},a_{2},...,a_{n}\right)  \right\Vert ^{p}\right)
^{\frac{1}{p}}\leq C\sup_{\left\Vert \varphi\right\Vert \leq1}\left(
{\displaystyle\sum\limits_{j=1}^{m_{1}}}
\left\vert \varphi\left(  x_{1,j},a_{2},...,a_{n}\right)  \right\vert
^{p}\right)  ^{\frac{1}{p}}%
\]
and it follows that $u$ is absolutely $p$-summing.
\end{proof}

The following two theorems are immediate consequences of the previous lemma
and of the respective linear results:

\begin{theorem}
[Dvoretzky-Rogers type theorem]\label{ppp}Let $Y$ be a Banach space. Every
continuous $n$-linear operator $T:Y\times\cdots\times Y\rightarrow Y$ is
factorable strongly $p$-summing if, and only if, $\dim Y<\infty.$
\end{theorem}

\begin{theorem}
[Lindenstrauss--Pe\l czy\'{n}ski type theorem]\label{uuu}Let $m$ be a positive
integer. If $X$ and $Y$ are infinite-dimensional Banach spaces, $X$ has an
unconditional Schauder basis and $\Pi_{FSt,1}(^{m}X;Y)=\mathcal{L}(^{m}X;Y)$
then $X=\ell_{1}$ and $Y$ is a Hilbert space.
\end{theorem}

For polynomials we have a natural version of Lemma \ref{sin}:

\begin{lemma}
\label{870}If every continuous $n$-homogeneous polynomial $P:X\rightarrow Y$
is factorable strongly $p$-summing, then every continuous linear operator
$u:X\rightarrow Y$ is absolutely $p$-summing.
\end{lemma}

\begin{proof}
Let $u:X\rightarrow Y$ be a continuous linear operator and $\varphi\in
X^{\ast}$ be a non-null linear functional and $a\in X$ be so that
$\varphi\left(  a\right)  =1$. Then $P\left(  x\right)  :=u\left(  x\right)
\varphi\left(  x\right)  ^{n-1}$ is factorable strongly $p$-summing. Thus
$\check{P}$ is factorable strongly $p$-summing. From the proof of Lemma
\ref{sin} we conclude that the linear operator $v:X\rightarrow Y$ defined by
$v(x)=\check{P}\left(  a,....,a,x\right)  $ is absolutely $p$-summing. But $v$
is a linear combination of $u\left(  a\right)  \varphi$ and $u;$ since
$u\left(  a\right)  \varphi$ is absolutely $p$-summing it follows that $u$ is
absolutely $p$-summing.
\end{proof}

An immediate consequence of the previous lemma is that the analogs of theorems
\ref{ppp} and \ref{uuu} work for polynomials. For instance:

\begin{theorem}
[Dvoretzky-Rogers type theorem for polynomials]Let $Y$ be a Banach space.
Every continuous $n$-homogeneous polynomial $P:Y\rightarrow Y$ is factorable
strongly $p$-summing if, and only if, $\dim Y<\infty.$
\end{theorem}

Given $P\in{\mathcal{P}}(^{m}X;Y)$ let us consider its transpose
$P^{t}:Y^{\ast}\rightarrow{\mathcal{P}}(^{m}X)$ given by $P^{t}(y^{\ast
}):=y^{\ast}\circ P$. Note that $P^{t}$ is a continuous linear operator. Let
$P^{tt}:{\mathcal{P}}(^{m}X)^{\ast}\rightarrow Y^{\ast\ast}$ be the transpose
of $P^{t}$. It is well known (see \cite[Theorem 2.21]{DJT})
that, if $Y=H$ is a Hilbert space then a continuous linear operator is
absolutely $1$-summing whenever its transpose is absolutely $p$-summing for
some $1\leq p<\infty$. Let us see that the analogous result is true for polynomials.

\begin{proposition}
Let $H$ be a Hilbert space and $P\in{\mathcal{P}}(^{m}X;H)$. If $P^{t}\in
\Pi_{p}(\hat{\otimes}_{\pi_{s}}^{m,s}X;H)$ for some $1\leq p<\infty$ then
$P\in{\mathcal{P}}_{FSt,1}(^{m}X;H)$.
\end{proposition}

\begin{proof}
From the equality  $P_{L,s}^{t}=\delta\circ P^{t}$, where
$\delta:{\mathcal{P}}(^{m}X)\rightarrow\left(  \hat{\otimes}_{\pi_{s}}^{m,s}X\right)  ^{\ast}$ is the canonical isomorphism,
it follows that $P_{L,s}^{t}$ is absolutely $p$-summing and then $P_{L,s}$ is
absolutely $1$-summing. By Theorem \ref{linearization} we conclude that $P$ is
factorable strongly $1$-summing.
\end{proof}

\begin{proposition}
Let $P\in{\mathcal{P}}(^{m}X;Y)$. Then $P\in{\mathcal{P}}_{FSt,p}(^{m}X;Y)$ if
and only if $P^{tt}\in\Pi_{p}({\mathcal{P}}(^{m}X)^{\ast};Y^{\ast\ast})$.
\end{proposition}

\begin{proof}
It is a consequence of Theorem \ref{linearization}, the fact that
 $P_{L,s}
^{tt}=P^{tt}\circ\delta^{t}$ and the analogous well known property for linear
operators (see \cite[Proposition 2.19]{DJT}).
\end{proof}

\section{Coherence and compatibility}

\label{coherence}

Let us denote the ideal of factorable strongly $p$-summing $n$-homogeneous
polynomials by $\mathcal{P}_{FSt,p}^{n}$, whereas $\Pi_{FSt,p}^{n} $ denotes
the ideal of factorable strongly $p$-summing $n$-linear operators. The notions
of coherent and compatible ideals of polynomials were introduced by Carando,
Dimant and Muro \cite{CDM09} in order to evaluate what polynomial approaches
preserve the spirit of a given operator ideal. Standard calculations show that
$\left(  \mathcal{P}_{FSt,p}^{n}\right)  _{n=1}^{\infty}$ is coherent and
compatible with $\Pi_{p}.$ Very recently, in \cite{rib}, the notions of
coherence and compatibility were extended to pairs of ideals of polynomials
and multi-ideals. It is also possible to show that $\left(  \mathcal{P}%
_{FSt,p}^{n},\Pi_{FSt,p}^{n} \right)  _{n=1}^{\infty}$ is coherent and
compatible with $\Pi_{p}.$

We have shown that $\mathcal{P}_{FSt,p}^{n}$ coincides with the composition
ideal with the absolutely $p$-summing operators. However, we cannot apply
\cite[Theorem~5.7]{rib} to get the coherence and compatibility as the topology
involved in that result comes from the multilinear operators space norm and it
does not coincide with $\|\cdot\|_{FSt,p}$ (see \cite{BoPeRuPRIMS2007}).
Despite of this, standard calculations also allow to get the following:

\begin{theorem}
The sequence $\left(  \mathcal{P}_{FSt,p}^{n},\Pi_{FSt,p}^{n} \right)
_{n=1}^{\infty}$ is coherent and compatible with $\Pi_{p}.$
\end{theorem}



\vspace*{1em}

\noindent[Daniel Pellegrino] Departamento de Matem\'atica, Universidade
Federal da Pa-ra\'iba, 58.051-900 - Jo\~ao Pessoa, Brazil, e-mail: dmpellegrino@gmail.com.

\medskip\noindent[Pilar Rueda] Departamento de An\'alisis Matem\'atico,
Universidad de Valencia, 46100 Burjassot - Valencia, Spain, e-mail: pilar.rueda@uv.es

\medskip

\noindent[Enrique A. S\'anchez P\'erez] Instituto Universitario de
Matem\'{a}tica Pura y Aplicada, Universidad Poli\'ecnica de Valencia, Camino
de Vera s/n, 46022 Valencia, Spain, e-mail: easancpe@mat.upv.es

\end{document}